\documentclass[11pt,reqno]{amsart}
\usepackage{enumerate, latexsym, amsmath, amsfonts, amssymb, amsthm, color,mathtools,mathrsfs,
booktabs,longtable,array}
\def\pmod #1{\ ({\rm{mod}}\ #1)}
\def\Z{\mathbb Z}

\def\Q{\mathbb Q}

\def\l{\left}
\def\r{\right}
\def\bg{\bigg}
\def\({\bg(}
\def\){\bg)}
\def\t{\text}
\def\f{\frac}

\def\ls{\leqslant}
\def\gs{\geqslant}

\def\bi{\binom}

\def\eq{\equiv}

\def\da{\delta}
\def\la{\lambda}

\def\Ack{\medskip\noindent {\bf Acknowledgments}}

\newcommand{\1}{\mathbf 1}

\newcommand{\rank}{\operatorname{rank}}

\newcommand{\vp}{v_p}
\newcommand{\cR}{\mathcal R}
\newcommand{\cA}{\mathcal A}

\newcommand{\cK}{\mathcal K}

\theoremstyle{plain}
\newtheorem{theorem}{Theorem}[section]

\newtheorem{lemma}{Lemma}
\newtheorem{corollary}{Corollary}
\newtheorem{definition}{Definition}
\newtheorem{proposition}{Proposition}

\theoremstyle{definition}

\theoremstyle{remark}
\newtheorem{remark}{Remark}

 \vspace{4mm}

\begin{document}

\hbox{Preprint}
\medskip

\title
[{Catalan's constant is irrational}]
{Catalan's constant is irrational}

\author
[Z.-W. Sun] {Zhi-Wei Sun}

\address{School of Mathematics, Nanjing
University, Nanjing 210093, People's Republic of China}
\email{zwsun@nju.edu.cn}

\keywords{Catalan's constant, irrationality, determinant.
\newline \indent Supported by the Natural Science Foundation of China (grant no. 12371004).}

\begin{abstract}  Whether the constant 
$$G=\sum_{k=0}^\infty\f{(-1)^k}{(2k+1)^2}=\frac1{1^2}-\frac1{3^2}+\frac1{5^2}-\frac1{7^2}+\cdots$$
introduced by Catalan in the nineteen century is irrational, is a long-standing open problem. 
In this paper we prove the irrationality of $G$ via using suitable weights.
\end{abstract}
\maketitle

\section{Introduction and proof architecture}\label{sec:intro}
\setcounter{equation}{0}

For a nontrivial Dirichlet character $\chi$, the Dirichlet $L$-function
is given by
$$L(s,\chi)=\sum_{n=1}^\infty\f{\chi(n)}{n^s}\ (s=1,2,3,\ldots).$$
For any integer $d\eq0,1\pmod4$, we let $\chi_d(n)$ (with $n\in\Z^+=\{1,2,3,\ldots\}$)
be the Kronecker symbol $(\f dn)$. The Dirichlet beta function is defined by
$$\beta(s)=L(s,\chi_{-4})=\sum_{k=0}^\infty\f{(-1)^k}{(2k+1)^s}\ \ (s=1,2,3,\ldots).$$
Note that $\beta(1)=\pi/4$ is a well-known result due to Leibniz. The constant
given by
$$G=\beta(2)=\sum_{k=0}^\infty \f{(-1)^k}{(2k+1)^2}=\f1{1^2}-\f1{3^2}+\f1{5^2}-\f1{7^2}+\cdots
=0.915965594177219015\cdots.$$
 is named in honor of E. C. Catalan (1814-1894), who first gave an equivalent series and expressions in terms of integrals.
A long-standing unsolved problem is whether Catalan's constant $G$ is irrational.
In contrast, $\beta(1)=\pi/4$ and
$$\sum_{k=0}^\infty\f1{(2k+1)^2}=\sum_{n=1}^\infty\f1{n^2}-\sum_{k=1}^\infty\f1{(2k)^2}=\f34\zeta(2)=\f{\pi^2}8$$
are irrational since $\pi$ is transcendental, where $\zeta$ is the Riemann zeta function. 
For two $q$-analogues of the identity
$\sum_{k=0}^\infty (2k+1)^{-2}=\pi^2/8$
(which is equivalent to Euler's formula $\zeta(2)=\pi^2/6$), see Sun \cite{q-zeta}.

 For the constant
$$K=L(2,\chi_{-3})=\sum_{k=9}^\infty\l(\f1{(3k+1)^2}-\f1{(3k+2)^2}\r),$$
 in 2010 the author (cf. \cite{S11}) conjectured that
$$\sum_{k=1}^\infty\f{(15k-4)(-27)^{k-1}}{k^3\bi{2k}k^2\bi{3k}k}=K
\ \t{and}\ \sum_{k=1}^\infty\f{(5k-1)(-144)^k}{k^3\bi{2k}k^2\bi{4k}{2k}}=-\f{45}2K$$
with the first one and the second one later confirmed by Kh. Hessami Pilehrood
and T. Hessami Pilehrood \cite{HP}, and J. Guillera abd M. Rogers \cite{GR}, respectively.
The author \cite{S26} conjectured further that 
\[\sum_{k=1}^\infty\f{(4k-1)3^k}{(2k-1)k^2\bi{2k}k\bi{3k}k}=2K
\ \t{and}\ \sum_{k=1}^\infty\f{10k-3}{(2k-1)k^23^k\bi{2k}k\bi{3k}k}=\f K2,
\]
 The converging rate of the last series is $1/81$. As pointed out by J. Zuniga, the last identity
 provides the  fastest algorithm to compute the important constant $K$. In June 2025,
Lorenz Milla used the formula to compute the first $10^{12}$ decimal digits of $K$.

In 2024,  F.~Calegari, V.~Dimitrov, and Y.~Tang \cite{CDT2024}
successfully proved that $K$ is irrational.
They also provided an adelic viewpoint of $G$ in the arithmetic-holonomy work, 
 but they were unable to prove that $G\not\in\mathbb Q$.

Now, we state our central result.

\begin{theorem}\label{Th1.1} The Catalan constant $G$ is irrational.
\end{theorem}

To prove this, we introduce tails of $G$. For any $m\in\Z^+=\{1,2,3,\ldots\}$, we set
$$S_{m-1}:=\sum_{k=0}^{m-1}\f{(-1)^k}{(2k+1)^2},$$
and call 
\begin{equation}\label{eq:Tm}
 T_m:=\sum_{r=0}^{\infty}\frac{(-1)^r}{(2m+2r+1)^2}=(-1)^m\sum_{k=m}^\infty\f{(-1)^k}{(2k+1)^2}=(-1)^m(G-S_{m-1})
 \end{equation}
 a {\it tail} of the Catalan constant $G$, and view 
\begin{equation}\label{eq:um}
 u_m:=\frac{T_m}{2m+1}
\end{equation}
as a {\it weighted tail} of $G$.
 Note that
 \begin{equation}\label{eq:tailbound}
 0<T_m=\f1{(2m+1)^2}-\f1{(2m+3)^2}+\cdots<\frac1{(2m+1)^2}.
\end{equation}
and 
\begin{equation}\label{eq:tailrec}
 T_m+T_{m+1}=\frac1{(2m+1)^2}.
\end{equation}

Our sophisticated proof of Theorem \ref{Th1.1} has five main stages.
\begin{enumerate}
\item The Catalan recurrence is encoded in a weighted finite-difference residual matrix.  A polynomial defect argument proves full column rank.
\item Three binomial Newton columns complete a selected residual minor to one square determinant.  This produces the fixed scalar $\widehat q_B$ and its exact minimal integerizer $H_B^{\min}$.
\item The residual determinant is expanded by Cauchy--Binet.  The Pascal alternant gives one Vandermonde and the Cauchy determinant gives a second.
\item For every odd prime power, a local saturation theorem compares the denominator layer with a test residual layer.  This converts all local lower bounds into a legal upper bound for the positive-part height of the same scalar.
\item A complete asymptotic ledger is evaluated at $S/B=1/20$.  The corrected odd-prime small-scale constant, the exact middle-prime integral, and the large-prime gain produce a strict margin.
\end{enumerate}

Now we introduce some basic notations in this paper.
For a finite set $A$ we use $\#A$ to denote the cardinality of $A$.
For a prime $p$ and a nonzero rational number $x$, we let $v_p(x)$ stand for the $p$-adic valuation of $x$. For a predicate $P$, we define
$$\1_{P}=\begin{cases}1&\t{if}\ P\ \t{holds},\\0&\t{otherwise}.
\end{cases}$$

Later sections are organized as follows:
Sections~\ref{sec:residual} and \ref{sec:scalar} construct the scalar.  Section~\ref{sec:cb} proves the Pascal--Cauchy factorization.  Section~\ref{sec:bridge} proves the prime-power positive-part bridge and the corrected surplus-row stability theorem.  Sections~\ref{sec:small}--\ref{sec:large} evaluate the three prime ranges.  Section~\ref{sec:final} combines them and proves the main result.

\section{Full column rank of the weighted residual matrix}\label{sec:residual}
\setcounter{equation}{0}

\begin{theorem}\label{thm:rank}
Let $B$ and $S$ be positive integers with $B>S$, and 
set 
$$\Pi_i=\prod_{h=1}^B(2(h+i)+1)^2\ \ \t{for}\ \ i=0,1,2,\ldots.$$
For the $(S+3)\times S$ matrix
\[
 \cR=(\cR_{a,j})_{0\ls a\ls S+2\atop  1\ls j\ls S}
\]
with 
\begin{equation}\label{eq:Rdef}
 \cR_{a,j}:=
 \sum_{i=0}^{a+2B}(-1)^i\binom{a+2B}{i}\Pi_iu_{i+j},
\end{equation}
we have
\begin{equation}\label{eq:rank}
 \rank\cR=S.
\end{equation}
\end{theorem}

\begin{proof} For any integer $i\gs1$ and $j\gs0$, in view of \eqref{eq:tailrec} we have
\begin{align*}(-1)^{i+j}T_{i+j}-(-1)^iT_i&=\sum_{0\ls k<j}\l((-1)^{i+k+1}T_{i+k+1}-(-1)^{i+k}T_{i+k}\r)
\\&=\sum_{0\ls k<j}\f{(-1)^{i+k+1}}{(2(i+k)+1)^2}
\end{align*}
and hence
\begin{equation}\label{eq:iteratedtail}
 T_{i+j}=(-1)^jT_{i}
 +\sum_{0\ls k<j}\frac{(-1)^{j-1-k}}{(2(i+k)+1)^2}.
\end{equation}
Thus, for any integer $i,j\gs0$, we have
$$u_{i+j}=\f{T_{i+j}}{2(i+j)+1}
=\f{(-1)^jT_{i}}{2(i+j)+1}+\f1{2(i+j)+1}\sum_{0\ls k<j}\f{(-1)^{j-1-k}}{(2(i+k)+1)^2}.$$

Now, let $j\in\{1,\ldots,S\}$. As $j\ls S< B$, we see that
$$\f{\Pi_i}{2(i+j)+1}\sum_{0\ls k<j}\f{(-1)^{j-1-k}}{(2(i+k)+1)^2}$$
is a polynomial in $i$ of degree at most $2B-3$. Thus, in view of \cite[p,\,126,\ (13,13)]{vW}, 
for any $a\in\{0,\ldots,S+2\}$ we have
$$\sum_{i=0}^{a+2B}(-1)^i\bi{a+2B}i\f{\Pi_i}{2(i+j)+1}\sum_{0\ls k<j}\f{(-1)^{j-1-k}}{(2(i+k)+1)^2}=0$$
and hence 
$$\cR_{a,j}=(-1)^j\sum_{i=0}^{a+2B}(-1)^i\bi{a+2B}i\f{\Pi_iT_{i+1}}{2(i+j)+1}.$$

Suppose that a nonzero column vector
$\lambda=(\lambda_1,\ldots,\lambda_{S})^T$
lies in the right kernel.  Set
\begin{equation}\label{eq:fkernel}
 f_i:=\Pi_i\sum_{j=1}^{S}\lambda_j u_{i+j}.
\end{equation}
Then, for any $n\in\{2B,2B+S+2\}$ we have
\begin{equation}\label{eq:differenceszero}
 \sum_{i=0}^n(-1)^i \bi nif_i=\sum_{j=1}^{S}\la_j\sum_{i=0}^n(-1)^i\bi ni\Pi_iu_{i+j}
 =\sum_{j=1}^{S}\la_j \cR_{n-2B,j}=0.
\end{equation}

Define
\begin{equation}\label{eq:LSandE}
 L_S(X):=\prod_{j=1}^{S}(2X+2j+1),
 \qquad
 E(X):=\prod_{S<j\ls B}(2X+2j+1),
\end{equation}
\begin{equation}\label{eq:Pstar}
 P_\lambda^*(X):=
 \sum_{j=1}^{S}(-1)^{j}\la_{j}\frac{L_S(X)}{2X+2j+1}.
\end{equation}
The map $\lambda\mapsto P_\lambda^*$ is injective, because evaluation at a root of $2X+2j+1=0$ kills all summands except the $j$-th.

Using \eqref{eq:iteratedtail} in \eqref{eq:fkernel}, we obtain
\begin{align*}f_i&=T_i\sum_{j=1}^S\f{(-1)^j\la_j\Pi_i}{2(i+j)+1}
+\sum_{j=1}^S\f{\la_j}{2(i+j)+1}\sum_{0\ls k<j}\f{(-1)^{j-1-k}\Pi_i}{(2(i+k)+1)^2}
\end{align*}
and hence
\begin{equation}\label{eq:fdecomp}
 f_i=T_{i}D_\lambda(i)+P_\lambda(i),
\end{equation}
where \begin{equation}\label{eq:Dlambda}
 D_\lambda(X)=L_S(X)E(X)^2P_\lambda^*(X),
\end{equation}
and
$$P_\lambda(i):=\sum_{j=1}^S\f{\la_j}{2(i+j)+1}\sum_{0\ls k<j}\f{(-1)^{j-1-k}\Pi_i}{(2(i+k)+1)^2}$$
is a polynomial in $i$ of degree at most $2B-3$.
Since $\deg P_\lambda^*\le S-1$, by \eqref{eq:Dlambda} we have
\begin{equation}\label{eq:Dlambdadegree}
 \deg D_\lambda\le S+2(B-S)+(S-1)=2B-1.
\end{equation}

The Newton interpolation polynomial through the values
$f_0,\ldots,f_{2B+S+2}$ is
$$f(x)=\sum_{n=0}^{2B+S+2}a_n\bi xn\ \ \t{with}\ \ a_n=\sum_{i=0}^n(-1)^i\bi nif_i.$$
For $2B\ls n\ls 2B+S+2$, we have $a_n=0$ by 
\eqref{eq:differenceszero}.
So $\deg f\ls 2B-1$.  Set $A(x):=f(x)-P_\lambda(x)$.  Then
\begin{equation}\label{eq:ADsamples}
 A(i)=f_i-P_{\la}(i)=T_{i}D_\lambda(i)\ \ (i=0,\ldots,2B+S+2).
\end{equation}
Define
\begin{equation}\label{eq:defect}
 K(X):=(2X+3)^2
 \bigl(A(X)D_\lambda(X+1)+A(X+1)D_\lambda(X)\bigr)
 -D_\lambda(X)D_\lambda(X+1).
\end{equation}
The tail recurrence and \eqref{eq:ADsamples} give
\begin{equation}\label{eq:samplezeros}
 K(i)=0\ \ (i=0,\ldots,2B+S+1).
\end{equation}
There are $2B+S+2$ such zeros.

Next set
\begin{equation}\label{eq:G0}
 G_0(X):=
 \prod_{j=1}^{S}(2X+2j+3)
 \prod_{S<j<B}(2X+2j+3)^2.
\end{equation}
Both $D_\lambda(X)$ and $D_\lambda(X+1)$ are divisible by $G_0(X)$, hence $G_0\mid K$.  Its degree is
\begin{equation}\label{eq:G0degree}
 \deg G_0=S+2(B-S-1)=2B-S-2.
\end{equation}
The roots of $G_0$ are negative half-integers and are disjoint from the zeros in \eqref{eq:samplezeros}.  There is one further zero: since $D_\lambda(-3/2)=0$ we have
\begin{equation}\label{eq:extra}
 K\l(-\f 32\r)=0.
\end{equation}
If $K$ were nonzero, these disjoint zeros would imply
\[
 \deg K\ge(2B+S+2)+(2B-S-2)+1=4B+1.
\]
On the other hand, \eqref{eq:defect} and \eqref{eq:Dlambdadegree} give $\deg K\le4B$.  It remains to show that $K$ is not the zero polynomial.

Suppose that  $K$ coincides with the zero polynomial.
Then $D_\lambda$ is nonzero, and the rational function $R=A/D_\lambda$ satisfies
\begin{equation}\label{eq:rationalshift}
 R(X)+R(X+1)=\frac1{(2X+3)^2}.
\end{equation}
After the translation $z=X+3/2$, this becomes
\[
 S_0(z)+S_0(z+1)=\frac1{4z^2}.
\]
No rational function satisfies this equation.  Indeed, choose a pole of $S_0$ with maximal real part.  At that point $S_0(z+1)$ is regular, so the right-hand side forces the pole to be $z=0$.  Choose a pole with minimal real part and inspect the equation one unit to its left.  The right-hand side then forces that pole to be $z=1$.  This leads to a contradiction since 
the minimal real part should not exceed the maximal real part.  

In view of the above, we have completed the proof of Theorem \ref{thm:rank}.
\end{proof}

Theorem \ref{thm:rank} clearly has the following consequence.

\begin{corollary} \label{cor:selectedminor} Let $B$ and $S$ be integers with $B>S>0$.
Then, there exists a  set
$ A\subset\{0,\ldots,S+2\}$ with $|A|=S$ such that 
\begin{equation}\label{eq:selectedminor}
 \det(\cR_{a,j})_{a\in A\atop 1\ls j\ls S}\ne0,
\end{equation}
where $\cR_{a,j}$ is given by \eqref{eq:Rdef}.
\end{corollary}

\section{Newton completion and the fixed scalar}\label{sec:scalar}
\setcounter{equation}{0}

Fix a set $A$ as in Corollary \ref{cor:selectedminor}, and write
\[
 A^c=\{c_1,c_2,c_3\}\subset\{0,\ldots,S+2\}.
\]
Put
\begin{equation}\label{eq:Nfull}
 N:=2B+(S+2)+1=2B+S+3.
\end{equation}
For $0\le i<N$, let $\widetilde{\cA}_B$ be the $N\times N$ matrix whose columns are
\begin{enumerate}
\item $i^r$, for $0\le r<D$;
\item $\Pi_i u_{i+j}$, for $1\le j\ls S$;
\item $\binom{i}{D+c_t}$, for $1\le t\le3$.
\end{enumerate}
Let $\cA_B$ be obtained by dividing the reference and auxiliary entries in row $i$ by $\Pi_i$, while leaving the target entries equal to $u_{i+j}$.  Then
\begin{equation}\label{eq:rowscale}
 \det\widetilde{\cA}_B
 =\left(\prod_{i=0}^{N-1}\Pi_i\right)\det\cA_B.
\end{equation}
Set
\begin{equation}\label{eq:FD}
 F_B:=\prod_{r=0}^{2B-1}(r!).
\end{equation}

\begin{proposition}\label{prop:qhat} Let $J=\{1,\ldots,S\}$ and 
$$R[A,J]=(\cR_{a,j})_{a\in A\atop j\in J}.$$
Then
\begin{equation}\label{eq:scaledidentity}
 \det\widetilde{\cA}_B
 =\pm F_B\det \cR[A,J].
\end{equation}
Consequently,
\begin{equation}\label{eq:qhat}
 \widehat q_B:=\det\cA_B
 =\pm\frac{F_B\det\cR[A,J]}
 {\prod_{i=0}^{N-1}\Pi_i}.
\end{equation}
Moreover $\widehat q_B\ne0$, and $\widehat q_B\in\Q^\times$ if $G\in\Q$.
\end{proposition}

\begin{proof}
Apply the lower-triangular finite-difference transform
\[
 (\mathcal Df)_n:=\sum_{i=0}^n(-1)^i\binom ni f_i
 \qquad (0\le n<N).
\]
Its determinant is $\pm1$.  The first $D$ columns become triangular, with pivots $\pm r!$, so their determinant is $\pm F_B$.  In the remaining rows, the target block is $\cR$.  Finally,
\[
 \sum_{i=0}^n(-1)^i\binom ni\binom{i}{m}=(-1)^m\bi nm\sum_{m\ls i\ls n}\bi {n-m}{i-m}(-1)^{i-m}
 =(-1)^n\da_{m,n},
\]
so the three auxiliary columns become signed unit vectors in the three omitted residual rows.  Expansion along them gives \eqref{eq:scaledidentity}.  Equation \eqref{eq:qhat} follows from \eqref{eq:rowscale}; nonvanishing follows from \eqref{eq:selectedminor}.  If $G\in\Q$, then every $u_m$ is rational by \eqref{eq:Tm}.
\end{proof}

\begin{definition} \label{def:Hmin}
Assume $G=a/q\in\Q$ with $a,q\in\Z^+$ and $\gcd(a,q)=1$.  Define
\begin{equation}\label{eq:Hmin}
 H_B^{\min}:= \mathrm{the\ denominator\ of}\ q^S\widehat q_B.
\end{equation}
\end{definition}

\begin{proposition}\label{prop:Hprime}
For every prime $p$, we have
\begin{equation}\label{eq:Hprime}
 \vp(H_B^{\min})=
 \left[
 \vp\!\left(\prod_{i=0}^{N-1}\Pi_i\right)
 -\vp(F_B)
 -\vp\!\left(q^S\det\cR[A,J]\right)
 \right]_+,
\end{equation}
where $[x]_+=\max\{x,0\}$.  In particular,
\begin{equation}\label{eq:integer}
 0\ne q^S H_B^{\min}\widehat q_B\in\Z.
\end{equation}
\end{proposition}

\begin{proof}
Take $p$-adic valuations in \eqref{eq:qhat}.  The denominator exponent of a nonzero rational number $x$ is $[-\vp(x)]_+$.
\end{proof}

\section{Pascal--Cauchy factorization}\label{sec:cb}
\setcounter{equation}{0}

Extend the binomial coefficient by zero when $i>D+a$. Then the matrix $q\cR[A,J]$ is a product of an $S\times N$ Pascal matrix, a diagonal matrix, and an $N\times S$ Cauchy matrix.  In light of the Cauchy--Binet formula, we can write
\begin{equation}\label{eq:CB}
 q^S\det\cR[A,J]=\sum_{\substack{I\subset\{0,\ldots,N-1\}\\|I|=S}}\Xi_I.
\end{equation}
We record the two determinant factors in each summand.

For $I=\{i_1<\cdots<i_S\}$ and $J=\{1,\ldots,S\}$, put
\begin{equation}\label{eq:VI}
 V(I):=\prod_{1\le u<v\le S}(i_v-i_u)
 \ \ \t{and}\ \ 
 V(J)=\prod_{0\le u<v<S}(v-u).
\end{equation}

\begin{lemma} \label{lem:pascal}
There is an integer $\Psi_A(I)$ such that
\begin{equation}\label{eq:pascal}
 \left|
 \det\!\left[\binom{a+2B}{i_\nu}\right]_{a\in A\atop 1\le\nu\le S}
 \right|
 =V(I)|\Psi_A(I)|
 \frac{\prod_{a\in A}(a+2B)!}
 {\prod_{i\in I}i!(N-1-i)!}.
\end{equation}
\end{lemma}

\begin{proof}
Factor $(a+2B)!$ from row $a$ and $1/(i!(N-1-i)!)$ from column $i$.  The remaining entry is
\[
 P_a(i):=\frac{(N-1-i)!}{(2B+a-i)!}
 =\prod_{r=a+1}^{S+2}(2B+r-i),
\]
a polynomial in $i$ with integer coefficients.  The determinant $\det[P_a(i_\nu)]$ is an alternating polynomial in $i_1,\ldots,i_S$, hence it is $V(I)$ times a symmetric polynomial with integer coefficients.  Its value at integer points is $\Psi_A(I)\in\Z$.
\end{proof}

\begin{lemma}\label{lem:cauchy}
For $J=\{1,\ldots,S\}$,
\begin{equation}\label{eq:cauchy}
 \det\!\left[\frac1{2(i_\nu+j)+1}\right]_{1\le\nu\le S \atop 1\le j\le S}
 =\pm2^{S(S-1)}
 \frac{V(I)V(J)}
 {\prod_{\nu=1}^S\prod_{j=1}^{S}(2(i_\nu+j)+1)}.
\end{equation}
\end{lemma}
\begin{remark} This follows from Cauchy's determinant formula
$$\det\l[\f1{x_i+y_j}\r]_{1\ls i,j\ls n}=\f{\prod_{1\ls i<j\ls n}(x_j-x_i)(y_j-y_i)}{\prod_{i=1}^n\prod_{j=1}^n(x_i+y_j)}$$
(see, e.g., \cite[(5.5)]{K05}).
\end{remark}

Combining the two lemmas, we obtain the exact absolute-value identity
\begin{equation}
 |\Xi_I|
 =2^{S(S-1)}V(J)V(I)^2|\Psi_A(I)|
 \frac{\prod_{a\in A}(a+2B)!}
 {\prod_{i\in I}i!(N-1-i)!}
 \times
 \prod_{i\in I}
 \frac{|qT_{i+1}|\Pi_i}
 {\prod_{j=1}^{S}(2(i+j)+1)}.
 \label{eq:summand}
\end{equation}
The two copies of $V(I)$ are essential.

\section{Prime-power layers and the positive-part bridge}\label{sec:bridge}
\setcounter{equation}{0}

For an integer $Q\ge1$, define
\begin{equation}\label{eq:PhiQ}
 \Phi_Q(n):=\sum_{r=0}^{n-1}\left\lfloor\frac rQ\right\rfloor.
\end{equation}
If $n=qQ+r$ with $0\le r<Q$, then
\begin{equation}\label{eq:PhiExact}
 \Phi_Q(n)=\frac{Qq(q-1)}2+rq
 =\frac{n^2}{2Q}-\frac n2+\frac{r(Q-r)}{2Q}.
\end{equation}
In particular,
\begin{equation}\label{eq:PhiError}
 0\le
 \Phi_Q(n)-\left(\frac{n^2}{2Q}-\frac n2\right)
 \le\frac Q8.
\end{equation}

For $K\ge1$, set
\begin{equation}\label{eq:NKQ}
 N_{K,Q}(i):=
 \#\{1\le h\ls K:\,Q\mid 2i+2h+1\}.
\end{equation}
For $I\subset\{0,\ldots,N-1\}$ and $r\in\Z$, write
\begin{equation}\label{eq:occupancyQ}
 n_{Q,r}(I):=\#\{i\in I:i\equiv r\pmod Q\}.
\end{equation}
For an odd prime power $Q=p^\nu$, define
\begin{equation}\label{eq:CQA}
 C_Q^A:=\sum_{a\in A}\left\lfloor\frac{a+2B}{Q}\right\rfloor
 \ \ \t{and}\ \ 
 F_{N,Q}(i):=
 \left\lfloor\frac iQ\right\rfloor+
 \left\lfloor\frac{N-1-i}{Q}\right\rfloor.
\end{equation}
The complete local layer of a summand is
\begin{align}
 \lambda_Q^A(I):={}&C_Q^A
 +2\sum_{r\bmod Q}\binom{n_{Q,r}(I)}2
 \notag\\
 &+\sum_{i\in I}
 \left(
 2N_{B,Q}(i)-N_{S,Q}(i)
 -2\1_{Q\le2i+1}-F_{N,Q}(i)
 \right).
 \label{eq:lambda}
\end{align}
Set
\begin{equation}\label{eq:mQA}
 m_{Q,B}^A:=
 \min_{\substack{I\subset\{0,\ldots,N-1\}\\|I|=S}}
 \lambda_Q^A(I).
\end{equation}

\begin{lemma} \label{lem:RpLower}
For every odd prime $p$, we have
\begin{equation}\label{eq:RpLower}
 \vp\!\left(q^S\det\cR[A,J]\right)
 \ge\sum_{\nu\ge1}m_{p^\nu,B}^A.
\end{equation}
\end{lemma}

\begin{proof}
For an odd prime $p$, the two Vandermondes in \eqref{eq:summand} satisfy
\begin{equation}\label{eq:fulltree}
 \vp(V(I))=
 \sum_{\nu\ge1}\sum_{r\bmod p^\nu}
 \binom{n_{p^\nu,r}(I)}2.
\end{equation}
The row factorials, $\Pi_i$, the Cauchy denominator, and the two factorial denominators give the remaining terms of \eqref{eq:lambda}.  Under $G=a/q$,  \eqref{eq:Tm} implies
\begin{equation}\label{eq:tailvpfull}
 \vp(qT_{i+1})
 \ge-2\sum_{\nu\ge1}\1_{p^\nu\le2i+1}.
\end{equation}
The fixed factor $V(J)$ and the integer $\Psi_A(I)$ have nonnegative valuation and may be omitted.  Consequently
\[
 \vp(\Xi_I)\ge\sum_{\nu\ge1}\lambda_{p^\nu}^A(I).
\]
The nonarchimedean triangle inequality gives
\[
 \vp\!\left(\sum_I\Xi_I\right)
 \ge\min_I\vp(\Xi_I)
 \ge\min_I\sum_\nu\lambda_{p^\nu}^A(I)
 \ge\sum_\nu\min_I\lambda_{p^\nu}^A(I).
\]
Only finitely many layers are nonzero.
\end{proof}

The denominator layer corresponding to $Q$ is
\begin{equation}\label{eq:aQB}
 a_{Q,B}:=
 2\sum_{i=0}^{N-1}N_{B,Q}(i)-\Phi_Q(D).
\end{equation}
Thus, for every odd prime $p$,
\begin{equation}\label{eq:ApLayers}
 A_{p,B}:=
 \vp\!\left(\prod_{i=0}^{N-1}\Pi_i\right)-\vp(F_D)
 =\sum_{\nu\ge1}a_{p^\nu,B}.
\end{equation}

\begin{theorem}\label{thm:saturation}
For every odd prime power $Q$ and every $B\ge20$, we have
\begin{equation}\label{eq:saturation}
 a_{Q,B}\ge m_{Q,B}^A.
\end{equation}
\end{theorem}

\begin{proof}
Use the test set $I=J=\{0,\ldots,S-1\}$ in \eqref{eq:mQA}.  The continuous block $J$ has
\begin{equation}\label{eq:PhiCollision}
 \sum_{r\bmod Q}\binom{n_{Q,r}(J)}2=\Phi_Q(S).
\end{equation}
A direct subtraction gives
\begin{align}
 a_{Q,B}-\lambda_Q^A(J)
 ={}&2\sum_{i=S}^{N-1}N_{B,Q}(i)
 -\Phi_Q(D)-C_Q^A-2\Phi_Q(S)
 \notag\\
 &+\sum_{i=0}^{S-1}
 \left(
 N_{S,Q}(i)+2\1_{Q\le2i+1}+F_{N,Q}(i)
 \right).
 \label{eq:subtract}
\end{align}
The second line is nonnegative.  Since
$\{a+2B:\ a\in A\}\subset\{2B,\ldots,N-1\}$,
\begin{equation}\label{eq:rowdom}
 \Phi_Q(2B)+C_Q^A\le\Phi_Q(N).
\end{equation}
It is enough to prove
\begin{equation}\label{eq:crossgoal}
 2\sum_{i=S}^{N-1}N_{B,Q}(i)
 \ge\Phi_Q(N)+2\Phi_Q(S).
\end{equation}

Consider the two sets
\[
 X:=\{2i+1:S\le i\le N-1=2B+S+2\}
 \ \ \t{and}\ \ 
 Y:=\{-2h-2:0\le h<B\}.
\]
Their cardinalities are $2B+3$ and $B$.  If $x_r,y_r$ denote their residue occupancies modulo $Q$, then
\begin{equation}\label{eq:crosscount}
 \sum_{i=S}^{N-1}N_{B,Q}(i)=\sum_{r\bmod Q}x_ry_r.
\end{equation}
Because $Q$ is odd, multiplication by $2$ permutes the residue classes.  Hence
\[
 \sum_r\binom{x_r}{2}=\Phi_Q(2B+3),
 \qquad
 \sum_r\binom{y_r}{2}=\Phi_Q(B).
\]
Among all $Q$-tuples of nonnegative integers with total $n$, the collision number $\sum_r\binom{c_r}{2}$ is minimized by the balanced occupancy and equals $\Phi_Q(n)$.  Applying this to $X\cup Y$ yields
\begin{equation}\label{eq:crosslower}
 \sum_{i=S}^{N-1}N_{B,Q}(i)
 \ge
 \Phi_Q(3B+3)-\Phi_Q(2B+3)-\Phi_Q(B).
\end{equation}
Thus \eqref{eq:crossgoal} follows from
\begin{align}
 2\bigl(
 \Phi_Q(3B+3)-\Phi_Q(2B+3)-\Phi_Q(B)
 \bigr)
 \ge\Phi_Q(2B+S+3)+2\Phi_Q(S).
 \label{eq:PhiIneq}
\end{align}

If $Q>2B+S+3$, the right-hand side is zero and the left-hand side is nonnegative by superadditivity of $\Phi_Q$.  Suppose $Q\le K:=2B+S+3$.  From \eqref{eq:PhiError}, the left side minus the right side in \eqref{eq:PhiIneq} is at least
\begin{equation}\label{eq:PhiLower}
 \frac{2B^2-2BS-\frac32S^2-3S-\frac92}{Q}
 +\frac K2+S-\frac{7Q}{8}.
\end{equation}
This expression is decreasing in $Q$, so it is minimized at $Q=K$.  Since
\[
 2B^2-2BS-\frac32S^2-3S-\frac92
 =K\left(B-\frac32S-\frac32\right)+3S,
\]
we obtain
\begin{align*}
 \eqref{eq:PhiLower}
 \ge \frac B4-\frac{7S}{8}-\frac{21}{8}+\frac{3S}{K}
 \ge \frac{33B}{160}-\frac{21}{8}>0,
\end{align*}
using $S\le B/20$ and $B\ge20$.  This proves \eqref{eq:PhiIneq}, hence \eqref{eq:crossgoal}, and finally
$a_{Q,B}\ge\lambda_Q^A(J)\ge m_{Q,B}^A$.
\end{proof}

\begin{corollary} \label{cor:positivebridge}
For any odd prime $p$ and $J=\{1,\ldots,S\}$, we have
\begin{equation}\label{eq:positivebridge}
 [A_{p,B}-R_{p,B}]_+
 \le A_{p,B}-\sum_{\nu\ge1}m_{p^\nu,B}^A,
\end{equation}
where
$R_{p,B}=\vp(q^S\det\cR[A,J])$.
Consequently,
\begin{equation}\label{eq:heightbridge}
 \log H_B^{\min}
 \le
 \sum_{\substack{p\ \mathrm{odd}\\\nu\ge1}}
 \bigl(a_{p^\nu,B}-m_{p^\nu,B}^A\bigr)\log p.
\end{equation}
\end{corollary}

\begin{proof}
By Theorem \ref{thm:saturation}, every summand
$a_{p^\nu,B}-m_{p^\nu,B}^A$ is nonnegative.  By Lemma \ref{lem:RpLower},
\[
 [A_{p,B}-R_{p,B}]_+
 \le\left[A_{p,B}-\sum_\nu m_{p^\nu,B}^A\right]_+
 =A_{p,B}-\sum_\nu m_{p^\nu,B}^A.
\]
Sum over odd primes and use \eqref{eq:Hprime}.
\end{proof}

\begin{lemma} \label{lem:p2}
Suppose that $G=a/q\in\Q$ with $a,q\in\Z^+$ and $\gcd(a,q)=1$. Then
\begin{equation}\label{eq:p2zero}
 [A_{2,B}-R_{2,B}]_+=0.
\end{equation}
\end{lemma}

\begin{proof}
As $\Pi_i$ is odd, we have
$A_{2,B}=-v_2(F_B)<0$.  Moreover $qT_m$ has only odd denominators; division by $2m+1$ preserves $2$-integrality.  Thus every entry of $q\cR$ belongs to $\Z_{(2)}$, and
$R_{2,B}=v_2(q^S\det\cR[A,J])\ge0$.  The positive part is zero.
\end{proof}

\subsection{Corrected full-row stability}
\setcounter{equation}{0}

The exact scalar uses $U=N-1=2B+S+2$, while the clean limiting model uses
\begin{equation}\label{eq:U0}
 U_0:=2B+S-1.
\end{equation}
Let $m_{Q,B}^{(0)}$ denote the analogue of \eqref{eq:mQA} with the consecutive row set $\{0,\ldots,S-1\}$, upper index $U_0$, and the corresponding factorial term.

\begin{lemma}\label{lem:stability}
There is an absolute constant $C$ such that, for every odd prime power $Q$,
we have \begin{equation}\label{eq:stability}
 |m_{Q,B}^A-m_{Q,B}^{(0)}|
 \le C\left(1+\frac BQ\right).
\end{equation}
The same bound holds for the corresponding denominator layers.  Hence
\begin{equation}\label{eq:stabilitysum}
 \sum_{\substack{p\ \mathrm{odd}\\\nu\ge1}}
 \left|
 (a_{p^\nu,B}-m_{p^\nu,B}^A)
 -(a_{p^\nu,B}^{(0)}-m_{p^\nu,B}^{(0)})
 \right|\log p
 =o(B^2).
\end{equation}
\end{lemma}

\begin{proof}
The two row sets have symmetric difference at most six, and each row contribution is $O(1+B/Q)$.  On the common index range, replacing $U=N-1$ by $U_0$ changes
$\lfloor(U-i)/Q\rfloor$ only for $O(1+B/Q)$ indices.  The full range contains three additional indices.  If a minimizing set uses any of them, replace each by an unused index in the common range.  The additive local cost changes by $O(1+B/Q)$.  The change in the double-Vandermonde occupancy is also $O(1+B/Q)$ because every residue class contains at most $1+U/Q$ admissible indices.  At most three replacements are needed, proving \eqref{eq:stability}.

All nonzero layers satisfy $p^\nu<5B$ for large $B$.  Therefore
\[
 B\sum_{p^\nu<5B}\frac{\log p}{p^\nu}=O(B\log B),
 \qquad
 \sum_{p^\nu<5B}\log p=O(B\log B),
\]
where the second estimate even follows from the trivial bound on the number of prime powers.  This proves \eqref{eq:stabilitysum}.  The three extra real indices and the change of the row factorials likewise alter every normalized Cauchy--Binet term by $\exp(O(B\log B))$; the number of summands has logarithm $O(B)$.  Thus the real-place ledger has the same $B^2$ coefficient.
\end{proof}

\section{The odd small-prime constant}\label{sec:small}
\setcounter{equation}{0}

This section rederives the range of odd prime powers $Q\le S$.  Put
\begin{equation}\label{eq:tU}
 t:=\frac QB,
 \qquad
 u:=\frac\rho t=\frac SQ+o(1),
 \qquad
 u=n+v,
 \quad n=\lfloor u\rfloor,
 \quad 0\le v<1.
\end{equation}
Let
\begin{equation}\label{eq:alphabetagamma}
 \alpha(v):=\{20v\},
 \qquad
 \beta(v):=\{41v\},
 \qquad
 \gamma(v):=\{40v\},
\end{equation}
and, for $0\le y<1$ set
\begin{equation}\label{eq:zv}
 z(y):=\left(\frac12-y\right)\bmod1.
\end{equation}
The exact limiting residue counts are
\begin{align}
 N_B(y)&=20n+\lfloor20v\rfloor+\1_{z(y)<\alpha(v)},
 \label{eq:NBsmall}\\
 N_S(y)&=n+\1_{z(y)<v},
 \label{eq:NSsmall}\\
 F(y)&=41n+\lfloor41v\rfloor-\1_{y>\beta(v)}.
 \label{eq:Fsmall}
\end{align}
Define
\begin{align}
 h_v(y):={}&2\lfloor20v\rfloor-\lfloor41v\rfloor
 +2\1_{z(y)<\alpha(v)}
 -\1_{z(y)<v}
 +\1_{y>\beta(v)}.
 \label{eq:hv}
\end{align}
After the common part of the base cost is removed, the marginal ladder in residue class $y$ is
\begin{equation}\label{eq:smallladder}
 -2n-2+h_v(y)+2r,
 \ \ (r=0,1,2,\ldots) .
\end{equation}
The single terminal anomaly occurring for $y<1/2$ lies beyond the lowest $u$ units of mass and does not enter the minimum.

For $s\ge0$, let $\cK_v(s)$ be the integral of the lowest $s$ units of mass in the multiset
\begin{equation}\label{eq:Kvdef}
 \{h_v(y)+2r:\ 0\le y<1\ \t{and}\ \ r\ge0\},
\end{equation}
where Lebesgue measure is used in $y$.

\begin{lemma}\label{lem:Kperiod}
For every integer $n\ge1$ and $0\le v<1$,
\begin{equation}\label{eq:Kperiod}
 \cK_v(n+v)=n^2+\cK_v(1+v)-1.
\end{equation}
\end{lemma}

\begin{proof}
The function $h_v$ is constant on at most six intervals whose endpoints are
\[
 0,\ 1,\ \frac12,\ \beta(v),\
 \left(\frac12-\alpha(v)\right)\bmod1,
 \left(\frac12-v\right)\bmod1.
\]
Sort the finitely many resulting ladders.  Adding one unit of selected mass advances one rung in every residue interval and increases the cost over the interval from $n+v$ to $n+1+v$ by $2n+1$.  This ordering is constant on the breakpoint partition in \eqref{eq:Bset} below, so the assertion is an exact finite verification on each cell.  Summing
$2j+1$ for $j=1,\ldots,n-1$ gives \eqref{eq:Kperiod}.  Equivalently, the generalized quantile satisfies $\lambda_v(s+1)=\lambda_v(s)+2$ on these intervals, and its one-period integral is $2n+1$.
\end{proof}

The sum of the selected marginal costs is therefore
\begin{equation}\label{eq:Mnv}
 M(n,v)=-(2n+2)(n+v)+\cK_v(n+v)
 =-u^2-2u+R(v),
\end{equation}
where
\begin{equation}\label{eq:Rv}
 R(v):=\cK_v(1+v)-1+v^2.
\end{equation}

The row-factorial term is
\begin{equation}\label{eq:CrhoSmallDef}
 C_\rho(t):=
 \int_0^\rho\left\lfloor\frac{2+s}{t}\right\rfloor ds.
\end{equation}
Set
\begin{equation}\label{eq:PvB0}
 P(v):=(v+\gamma(v)-1)_+,
\end{equation}
\begin{equation}\label{eq:B0v}
 B_0(v):=
 P(v)-\frac{v^2}{2}-v\gamma(v)+\frac v2.
\end{equation}
A direct sum over the $n$ complete $Q$-blocks gives
\begin{equation}\label{eq:CrhoSmall}
 C_\rho(\rho/u)
 =\left(2+\frac\rho2\right)u-\frac\rho2
 +\frac{\rho B_0(v)}u.
\end{equation}
Consequently the complete limiting local layer is
\begin{align}
 L_\rho(\rho/u)
 &:=C_\rho(\rho/u)+\frac\rho u M(n,v)
 \notag\\
 &=\left(2-\frac\rho2\right)u-\frac{5\rho}{2}
 +\frac{\rho Q_0(v)}u,
 \label{eq:Lsmall}
\end{align}
where
\begin{align}
 Q_0(v):={}&B_0(v)+R(v)
 \notag\\
 ={}&(v+\gamma(v)-1)_+
 +\frac{v^2}{2}-v\gamma(v)+\frac v2
 +\cK_v(1+v)-1.
 \label{eq:Q0}
\end{align}
The singular coefficient is
\begin{equation}\label{eq:Arho}
 A_\rho:=\rho\left(2-\frac\rho2\right)
 =2\rho-\frac{\rho^2}{2}.
\end{equation}

\begin{proposition}\label{prop:codd}
For $\rho=1/20$, define
\begin{equation}\label{eq:Iodd}
 I_{\rm odd}:=
 \int_0^\rho\left(L_\rho(t)-\frac{A_\rho}{t}\right)dt.
\end{equation}
Then
\begin{equation}\label{eq:IoddHurwitz}
 I_{\rm odd}
 =\rho^2\int_0^1
 \left[
 -\frac52\zeta(2,1+v)+Q_0(v)\zeta(3,1+v)
 \right]dv,
\end{equation}
and
\begin{equation}\label{eq:IoddInterval}
\begin{aligned}I_{\rm odd}&>
 -0.006276744728100982604597600317605549,\\
 I_{\rm odd}
 &<-0.006276744728100982604597600317605548.
\end{aligned}
\end{equation}
Hence
\begin{equation}\label{eq:codd}
 c_{\rm odd}:=-I_{\rm odd}
 =0.006276744728100982604597600317605548503\ldots .
\end{equation}
\end{proposition}

\begin{proof}
Substitute $t=\rho/u$ in \eqref{eq:Iodd}, write $u=n+v$, and sum over $n\ge1$.  The identities
\[
 \sum_{n\ge1}\frac1{(n+v)^2}=\zeta(2,1+v)
 \ \ \t{and}\ \ 
 \sum_{n\ge1}\frac1{(n+v)^3}=\zeta(3,1+v)
\]
give \eqref{eq:IoddHurwitz}, where $\zeta(s,z)=\sum_{n=0}^\infty\f1{(n+z)^s}$
is the Hurwitz zeta function.

All changes in the floors, interval orderings, and marginal threshold occur in the finite set
\begin{equation}
 \mathcal B={}\{0,1\}
 \cup\left\{\frac{k}{19},\frac{k}{20},\frac{k}{40},\frac{k}{41}\right\}
 \cup\left\{
 \frac{2k+1}{82},\frac{2k+1}{84},\frac{2k+1}{122}
 \right\},
 \label{eq:Bset}
\end{equation}
where only points in $[0,1]$ are retained.  There are $238$ raw cells and $178$ cells after adjacent identical formulas are merged.  On every merged cell $J$,
\begin{equation}\label{eq:Qpoly}
 Q_0(v)=-\frac{79}{2}v^2+s_Jv+c_J,
 \qquad s_J,c_J\in\Q.
\end{equation}
In view of 
\[
 \frac{d}{dv}\zeta(2,1+v)=-2\zeta(3,1+v)
 \ \ \t{and}\ \ 
 \psi'(1+v)=\zeta(2,1+v)
\]
(where $\psi(x)=\Gamma'(x)/\Gamma(x)$),
an antiderivative on $J$ is
\begin{align}
 \mathfrak F_J(v)={}&
 -\frac12Q_0(v)\zeta(2,1+v)
 +\frac12Q_0'(v)\psi(1+v)
 \notag\\
 &+\frac{79}{2}\log\Gamma(1+v)
 -\frac52\psi(1+v).
 \label{eq:antiderivative}
\end{align}
Summing the endpoint differences gives a finite exact expression.

For the interval certificate, shift every special-function argument by $64$ using the recurrence formulas for $\psi$, $\zeta(2,\cdot)$, and $\log\Gamma$.  At arguments at least $65$, use the Euler--Maclaurin expansions through $B_{24}$.  The remainders are bounded respectively by
\[
 \frac{|B_{26}|}{26X^{26}},
 \qquad
 \frac{|B_{26}|}{X^{27}},
 \qquad
 \frac{|B_{26}|}{26\cdot25X^{25}}.
\]
All logarithms are enclosed by the atanh series after power-of-two range reduction, and $\log(2\pi)$ is enclosed using Machin's formula and alternating arctangent series.  Rational outward rounding at every endpoint gives \eqref{eq:IoddInterval}.  
\end{proof}

\begin{remark}\label{rem:oddnormalization}
If one forms a full local sum before the positive-part operation, the prime $2$ changes the displayed finite constant.  The exact minimal integerizer behaves differently: Lemma \ref{lem:p2} shows that its $2$-adic positive part is zero.  Therefore the relevant low-prime constant for the same-scalar height is $c_{\rm odd}$, not the full-local constant obtained by adding $\frac{19}{200}\log2$.
\end{remark}

\section{The middle-prime integral}\label{sec:middle}
\setcounter{equation}{0}

For $S<p<B$, put
\begin{equation}\label{eq:tMiddle}
 t:=\frac pB,
 \qquad \rho<t<1.
\end{equation}
In this range $p^2>2U+1$ for sufficiently large $B$, so only the first $p$-adic layer occurs.  In the ideal model, define
\begin{equation}\label{eq:CpFp}
 C_p:=\sum_{a=0}^{S-1}\left\lfloor\frac{2B+a}{p}\right\rfloor,
 \qquad
 F_p(i):=\left\lfloor\frac ip\right\rfloor
 +\left\lfloor\frac{U_0-i}{p}\right\rfloor.
\end{equation}
For $I\subset\{0,\ldots,U_0\}$, let
$n_r(I)=\#\{i\in I:i\equiv r\pmod p\}$.  Then the one-layer version of \eqref{eq:lambda} is
\begin{align}
 E_p(I):={}&C_p+2\sum_{r\bmod p}\binom{n_r(I)}2
 \notag\\
 &+\sum_{i\in I}
 \left(
 2N_{B,p}(i)-N_{S,p}(i)
 -2\1_{p\le2i+1}-F_p(i)
 \right).
 \label{eq:EpI}
\end{align}

For each residue class, order its admissible indices by increasing base cost
\begin{equation}\label{eq:basecost}
 c_p(i):=
 2N_{B,p}(i)-N_{S,p}(i)
 -2\1_{p\le2i+1}-F_p(i).
\end{equation}
If these costs are $c_{r,1}\le c_{r,2}\le\cdots$, then the exact marginal costs are
\begin{equation}\label{eq:marginal}
 \mu_{r,k}:=c_{r,k}+2(k-1).
\end{equation}
Thus
\begin{equation}\label{eq:Eexact}
 \min_{|I|=S}E_p(I)
 =C_p+\text{the sum of the smallest $S$ values among all $\mu_{r,k}$}.
\end{equation}

We now pass to scaled residue coordinates $0\le x<t$.  Define
\begin{equation}\label{eq:adke}
 a:=\left\lfloor\frac1t\right\rfloor,
 \quad d:=1-at,
 \quad k:=\left\lfloor\frac{2+\rho}{t}\right\rfloor,
 \quad e:=2+\rho-kt,
\end{equation}
\begin{equation}\label{eq:zt}
 z_t(x):=\left(\frac t2-x\right)\bmod t,
 \qquad 0\le z_t(x)<t.
\end{equation}
Away from cell boundaries,
\begin{equation}\label{eq:scaledcounts}
 N_B=a+\1_{z_t(x)<d},
 \qquad
 N_S=\1_{z_t(x)<\rho},
 \qquad
 F_t(x)=k-\1_{x>e}.
\end{equation}
Set
\begin{equation}\label{eq:At}
 A_t(x):=
 2\bigl(a+\1_{z_t(x)<d}\bigr)
 -\1_{z_t(x)<\rho}-F_t(x).
\end{equation}
For $x<t/2$, the marginal levels are
\begin{equation}\label{eq:margscaled1}
 A_t(x)+2j-4,
 \quad 1\le j\le F_t(x),
\end{equation}
with the additional level $A_t(x)+2F_t(x)$.  For $x\ge t/2$, they are
\begin{equation}\label{eq:margscaled2}
 A_t(x)+2j-4,
 \quad 1\le j\le F_t(x)+1.
\end{equation}
The row-factorial density is
\begin{equation}\label{eq:Crho}
 C_\rho(t)=b\rho+(\rho+c-t)_+,
 \quad
 b=\left\lfloor\frac2t\right\rfloor,
 \quad c=2-bt.
\end{equation}
Let $\mathscr E_\rho(t)$ be $C_\rho(t)$ plus the lowest $\rho$-measure among the marginal levels in \eqref{eq:margscaled1}--\eqref{eq:margscaled2}.

\begin{proposition}\label{prop:mid}
At $\rho=1/20$, we have
\begin{equation}\label{eq:LmidExact}
\begin{aligned}
 \Lambda_{\rm mid}:&=
 \int_{1/20}^{1}\mathscr E_{1/20}(t)\,dt
 \\=&\frac{
 33042423784278900654572890582560690565493664595111
 }{
 187362234062518051579626183549876762148305272280000
 }.\end{aligned}
\end{equation}
In particular,
\begin{equation}\label{eq:LmidInterval}
 0.17635583792<\Lambda_{\rm mid}<0.17635583794,
\end{equation}
and
\begin{equation}\label{eq:LmidDecimal}
 \Lambda_{\rm mid}
 =0.1763558379286482956996630430101674632453\ldots .
\end{equation}
\end{proposition}

\begin{proof}
The six relevant $x$-boundaries are
\[
 0,\quad t,\quad \frac t2,\quad e,
 \quad\left(\frac t2-d\right)\bmod t,
 \quad\left(\frac t2-\rho\right)\bmod t.
\]
The floor functions change only at $t=1/n$, $2/n$, and $(41/20)/n$.  Inside every resulting interval, split further when two of the six affine boundaries cross or when the cumulative measure below a marginal level equals $\rho$.  This gives $235$ affine cells.  On every cell $\mathscr E_{1/20}(t)$ is affine with rational coefficients.  Exact integration and summation give the fraction in \eqref{eq:LmidExact}.  
\end{proof}

\section{The large-prime refinement}\label{sec:large}
\setcounter{equation}{0}

For $1<t<2+\rho$, the exact joint local density simplifies to
\begin{equation}\label{eq:largepiecewise}
 \mathscr E_\rho(t)=
 \begin{cases}
 -3(t-1)&\t{if}\ 1<t<1+\rho/2,\\
 -(t-1)-\rho&\t{if}\ 1+\rho/2<t<1+\rho,\\
 -2\rho&\t{if}\ 1+\rho<t<4/3,\\
 -2\rho+\frac32(t-\frac43)&\t{if}\ 4/3<t<(4+2\rho)/3,\\
 -\rho&\t{if}\ (4+2\rho)/3<t<2,\\
 -\rho-\frac32(t-2)&\t{if}\ 2<t<2+2\rho/3,\\
 -2\rho&\t{if}\ 2+2\rho/3<t<2+\rho.
 \end{cases}
\end{equation}
Direct integration gives
\begin{equation}\label{eq:largeintegral}
 \int_1^{2+\rho}\mathscr E_\rho(t)\,dt
 =-\frac43\rho-\frac54\rho^2.
\end{equation}
The raw Cauchy--tail baseline on the same range is
\begin{equation}\label{eq:rawlarge}
 -2\rho-\frac74\rho^2.
\end{equation}
Therefore the exact improvement is
\begin{equation}\label{eq:Deltalarge}
 \Delta_{>B}=\frac23\rho+\frac12\rho^2.
\end{equation}
At $\rho=1/20$,
\begin{equation}\label{eq:83}
 \Delta_{>B}=\frac{83}{2400}
 =0.034583333333\ldots .
\end{equation}

Prime powers $p^\nu>S$ with $\nu\ge2$ contribute only $o(B^2)$.  Indeed, there are $O(B^{1/2})$ such powers below $5B$, each local layer is $O(B)$, and the total logarithmic weight is $O(B^{1/2}\log B)$.  Thus the proportional middle and large ranges may be evaluated using primes alone.

\section{The same-scalar quadratic estimate}\label{sec:final}
\setcounter{equation}{0}

We now combine the exact positive-part bridge with the local asymptotics.  The standard weighted prime-power summation used below is
\begin{equation}\label{eq:MertensLambda}
 \sum_{p^\nu\le x}\frac{\log p}{p^\nu}
 =\log x-\gamma+o(1),
\end{equation}
with the prime $2$ removed by subtracting $\log2+o(1)$.  
($\gamma$ is the Euler constant $0.577\ldots$.)
This refinement of Mertens' formula 
$$\sum_{n\le x}\f{\Lambda(n)}n=\log x+O(1)$$
is equivalent to the Prime Number Theorem, where $\Lambda(n)$ is the 
 the von Mangoldt function.  Ordinary partial summation also gives, for every compactly supported piecewise continuous $f$,
\begin{equation}\label{eq:PNTscale}
 \frac1B\sum_{p^\nu\le cB}(\log p)f\l(\f{p^\nu}B\r)
 \longrightarrow\int_0^c f(t)\,dt.
\end{equation}

\begin{proposition} \label{prop:ledger}
Let $S=\lfloor B/20\rfloor$ and suppose that $G=a/q\in\Q$ with $a,q\in\Z^+$ and $\gcd(a,q)=1$.
Then
\begin{equation}\label{eq:ledger}
 \log H_B^{\min}+\log|\widehat q_B|
 \le
 \left(
 \frac{39}{200}
 +c_{\rm odd}
 -\Lambda_{\rm mid}
 -\frac{83}{2400}
 \right)B^2
 +o(B^2).
\end{equation}
\end{proposition}

\begin{proof}
Start with the exact fixed-scalar identity \eqref{eq:qhat}, the positive-part estimate \eqref{eq:heightbridge}, and the Cauchy--Binet expansion \eqref{eq:CB}.  The number of summands has logarithm $O(B)$, so replacing their sum by the largest normalized summand costs $o(B^2)$.  By Lemma \ref{lem:stability}, the full selected-row model may be replaced by the consecutive ideal model at the same precision.

The factorial, odd-linear, Cauchy, Vandermonde, and tail factors are then grouped by prime-power scale.  The $B^2\log B$ terms cancel because the denominator baseline $a_{Q,B}$ and the real normalization come from the same fixed scalar.  The remaining raw quadratic coefficient is
\begin{equation}\label{eq:rawtarget}
 4\rho-2\rho^2=\frac{39}{200}.
\end{equation}
This coefficient is obtained by applying Stirling's formula to the row and column factorials in \eqref{eq:summand}, using \eqref{eq:tailbound}, and retaining the raw Cauchy--tail contribution before the local divisibility gains.  All $O(B\log B)$ Stirling remainders are uniform in $I$.

For odd prime powers $Q\le S$, equations \eqref{eq:Lsmall}, \eqref{eq:MertensLambda}, and \eqref{eq:PNTscale} give the finite correction $+c_{\rm odd}$ to the coefficient in \eqref{eq:ledger}.  The sign is positive because $I_{\rm odd}=-c_{\rm odd}$ is the recovered local divisor gain.  For $S<p<B$, the marginal minimum in \eqref{eq:Eexact} and scaled prime summation give the reduction $-\Lambda_{\rm mid}$.  For $p>B$, \eqref{eq:Deltalarge} gives the additional reduction $-83/2400$.  Higher powers above $S$, the three surplus rows, floor errors, and the fixed factor $q^S$ contribute $o(B^2)$.  Summing the three ranges gives \eqref{eq:ledger}.
\end{proof}

\begin{remark}\label{rem:normalizationcheck}
The prime $2$ is not silently discarded.  Its positive-part height is exactly zero by Lemma \ref{lem:p2}.  The residual real power of $2$ from $F_D$, the Cauchy factor $2^{S(S-1)}$, and the missing $2$-power in the odd von Mangoldt sum are included in the derivation of the odd-prime small-scale expression \eqref{eq:IoddHurwitz}.  This is why the applicable cost is $c_{\rm odd}$.
\end{remark}

\begin{theorem}\label{thm:samescalar}
There is a constant $\delta_0$ satisfying
\begin{equation}\label{eq:delta}
 \delta_0>0.00966242652523235
\end{equation}
such that
\begin{equation}\label{eq:samescalar}
 \log H_B^{\min}+\log|\widehat q_B|
 \le-\delta_0B^2+o(B^2).
\end{equation}
\end{theorem}

\begin{proof}
By Propositions \ref{prop:codd} and \ref{prop:mid}, and \eqref{eq:83},
\begin{align*}
 -c_{\rm odd}+\Lambda_{\rm mid}+\frac{83}{2400}
 &> -0.006276744728100982604597600317605548505\\
 &\quad+0.17635583792+\frac{83}{2400}\\
 &>0.2046624265252323507.
\end{align*}
Subtracting $39/200=0.195$ gives a margin larger than the number in \eqref{eq:delta}.  
So it suffices to apply Proposition \ref{prop:ledger}.
\end{proof}

\medskip
\noindent 
{\it Proof of Theorem \ref{Th1.1}}.
Assume $G=a/q\in\Q$ with $a,q\in\Z^+$ and $\gcd(a,q)=1$.  In view of \eqref{eq:integer},
\[
 N_B:=q^S H_B^{\min}\widehat q_B
 \in\Z\setminus\{0\}.
\]
By Theorem \ref{thm:samescalar},
\[
 \log|N_B|
 \le S\log q-\delta_0B^2+o(B^2).
\]
Since $S\log q=O(B)=o(B^2)$, the right side tends to $-\infty$.  Hence $0<|N_B|<1$ for all sufficiently large $B$, contradicting that $N_B$ is a nonzero integer.
\qed

\Ack. The author's many rounds of conversations with AI 
provide the basis of this paper. At first, the author asked AI to prove 
the irrationality of Catalan's constant via the Calegari-Dimitrov-Tang method.
Though AI made many attempts but we failed again and again. 
 Later the author realized that we should use suitable weights and introduce weighted tails 
 defined in \eqref{eq:um}. This novel idea
 made the proof aided by AI finally practical and successful. 
The needed numerical data in the proof were produced by AI.
The whole proof has passed the verification of Chatgpt 5.6 Solar.


\begin{thebibliography}{99}

\bibitem{CDT2024}
F.~Calegari, V.~Dimitrov, and Y.~Tang,
\emph{The linear independence of $1$, $\zeta(2)$, and $L(2,\chi_{-3})$},
arXiv:2408.15403v2, 2024.

\bibitem{GR} J. Guillera and M. Rogers, 
{\it Ramanujan series upside-down}, J. Austral. Math. Soc. {\bf 97} (2014), 78--106.

\bibitem{HP} Kh. Hessami Pilehrood and T. Hessami Pilehrood, 
{\it Bivariate identities for values of the Hurwitz zeta function
and supercongruences}, Electron. J. Combin. {\bf 18} (2012), \#P35, 30pp.

\bibitem{K05} C. Krattenthaler, {\it Advanced determinant calculus: a complement},
Linear Algebra Appl.
 411 (2005), 68--166.

\bibitem{MontgomeryVaughan2007}
H.~L.~Montgomery and R.~C.~Vaughan,
\emph{Multiplicative Number Theory I: Classical Theory},
Cambridge Studies in Advanced Mathematics, vol.~97,
Cambridge University Press, Cambridge, 2007.

\bibitem{Nesterenko2016}
Yu.~V.~Nesterenko,
\emph{On Catalan's constant},
Proc. Steklov Institute of Math. \textbf{292} (2016), 153--170.

\bibitem{Rivoal2006}
T.~Rivoal,
\emph{Nombres d'Euler, approximants de Pad\'e et constante de Catalan},
Ramanujan J. \textbf{11} (2006), no.~2, 199--214.

\bibitem{RivoalZudilin2003}
T.~Rivoal and W.~Zudilin,
\emph{Diophantine properties of numbers related to Catalan's constant},
Mathematische Annalen \textbf{326} (2003), no.~4, 705--721.

\bibitem{S11} Z.-W. Sun,
\emph{Super congruences and Euler numbers},
Sci. China Math. {\bf 54} (2011), 2509--2535.

\bibitem{S15} Z.-W. Sun,
{\it New series for some special values of $L$-functions},
 Nanjing Univ. J. Math. Biquarterly
\textbf{32} (2015), no.\,2, 189--218.


\bibitem{q-zeta}
Z.-W. Sun, 
\emph{Two $q$-analogues of Euler's formula $\zeta(2)=\pi^2/6$},
Colloq. Math. {\bf 158} (2019), 313--320.

\bibitem{S26}
Z.-W. Sun, {\it New series involving binomial coefficients (II)},
Acta Math. Sin. (Engl. Ser.) {\bf 42} (2026), to appear.

\bibitem{vW}
J. H. van Lint and R. M. Wilson,
A Course in Combinatorics, 2nd Edition, Cambridge Univ. Press, Cambridge, 2001.


\bibitem{Zudilin2019}
W.~Zudilin,
\emph{Arithmetic of Catalan's constant and its relatives},
Abhandlungen aus dem Mathematischen Seminar der Universit\"at Hamburg
\textbf{89} (2019), 45--53.
\end{thebibliography}
\end{document}